\documentclass[a4paper,12pt]{amsart}  
\usepackage{amsmath,amssymb,amsthm}     
\usepackage[utf8]{inputenc}
\usepackage{enumitem}
\theoremstyle{plain}      
 
\newtheorem{thm}{Theorem}[section]     
\newtheorem{theorem}[thm]{Theorem}     
\newtheorem{cor}[thm]{Corollary}     
\newtheorem{corollary}[thm]{Corollary}     
     
\newtheorem{lemma}[thm]{Lemma}     
\newtheorem{prop}[thm]{Proposition}     
\newtheorem{proposition}[thm]{Proposition}

\theoremstyle{definition}      
\newtheorem{defn}[thm]{Definition}

\newtheorem{remark}[thm]{Remark}

\DeclareMathAlphabet{\doba}{U}{msb}{m}{n}         
\gdef\mC{\doba{C}}

\gdef\mR{\doba{R}}

\def\cL{\mathcal{L}}

\def\di{{\rm d}}

\let\<\langle 
\let\>\rangle

\newcommand{\definedas}{\mathrel{\raise.095ex\hbox{\rm :}\mkern-5.2mu=}}

\title{}

\author{Farid Madani, Andrei Moroianu, Mihaela Pilca}

\address{Farid Madani\\Fakult\"at f\"ur Mathematik\\
	Universit\"at Regensburg\\Universit\"atsstr. 31 
	D-93040 Regensburg, Germany}
\email{madani@mathematik.uni-regensburg.de}

\address{Andrei Moroianu \\ Laboratoire de Mathématiques d’Orsay, Univ. Paris-Sud, CNRS, Université Paris-Saclay, 91405 Orsay, France}
\email{andrei.moroianu@math.cnrs.fr}

\address{Mihaela Pilca\\Fakult\"at f\"ur Mathematik\\
Universit\"at Regensburg\\Universit\"atsstr. 31 
D-93040 Regensburg, Germany}
\email{mihaela.pilca@mathematik.uni-regensburg.de}

\subjclass[2010]{53A30, 53B35, 53C25, 53C29, 53C55}

\keywords{locally conformally K\"ahler structure, holomorphic Lee vector field, Vaisman structure, lcK structure with potential}

\begin{document}   

\title[LcK structures  with holomorphic Lee vector field]{LcK structures  with holomorphic Lee vector field on Vaisman-type manifolds}

\begin{abstract}
We give a complete description of all locally conformally K\"ahler structures with holomorphic Lee vector field on a compact complex manifold of Vaisman type. This provides in particular examples of such structures whose Lee vector field is not homothetic to the Lee vector field of a Vaisman structure. More generally, dropping the condition of being of Vaisman type, we show that on a compact complex manifold, any lcK metric with potential and with holomorphic Lee vector field admits a potential which is positive and invariant along the anti-Lee vector field.
\end{abstract}
\maketitle

\section{Introduction}

A Hermitian metric $g$ on a complex manifold $(M,J)$ is called locally conformally K\"ahler (in short, lcK) if  around  any  point  in $M$,  the  metric $g$ can  be  conformally  rescaled  to  a K\"ahler metric. This condition is equivalent to the existence of a closed $1$-form $\theta$ such that $\di\Omega=\theta\wedge\Omega$, where $\Omega$ denotes the fundamental $2$-form defined as $\Omega(\cdot, \cdot):=g(J\cdot, \cdot)$. The $1$-form $\theta$ is called the Lee form and its metric dual is called the Lee vector field. In this paper we assume that $\theta\neq 0$, \emph{i.e.} $(J,g,\Omega)$ is not K\"ahler. 

A special class of lcK structures is represented by the so-called Vaisman structures, defined by the property that the Lee form is non-zero and parallel with respect to the Levi-Civita connection of $g$. It is known that a Vaisman structure on a complex manifold is uniquely determined, up to a positive constant, by its Lee form, via the following identity: $$\Omega=\frac{1}{|\theta|^2}(\theta\wedge J\theta-\di J\theta).$$
On a compact complex manifold of Vaisman type, \emph{i.e.} admitting at least one Vaisman metric, the Lee vector fields of all Vaisman structures are holomorphic, and coincide up to a positive multiplicative constant. This fact was originally obtained in \cite{ts}, but for the reader's convenience we give below an alternative proof. 

In \cite{mmo}, A. Moroianu, S. Moroianu and L. Ornea proved that a compact lcK manifold with holomorphic Lee vector field is Vaisman if the Lee vector field either has constant norm or is divergence-free. Moreover, they also construct examples of non-Vaisman lcK structures with holomorphic Lee vector field on a compact  manifold of Vaisman type. These lcK structures, however, have the same Lee vector field as the Vaisman structures. More recently, F. Belgun \cite{b} constructed examples of lcK manifolds with holomorphic Lee vector field on compact complex manifolds which are not of Vaisman type. 

In this paper we describe all lcK structures with holomorphic Lee vector field on compact complex manifolds of Vaisman type. In particular, this description provides examples of such structures whose Lee vector field is not homothetic to the common Lee vector field of the Vaisman structures (see Remark~\ref{remexpl}) and which were not known to exist up to date. We also describe all holomorphic vector fields on Vaisman-type manifolds which can be obtained as Lee vector fields of lcK structures.

After introducing the needed notions and notation in a preliminary section, we will recall  in Section \ref{sectdeform} the description, due to K. Tsukada \cite{ts}, of the space of  all $1$-forms which occur as Lee forms of Vaisman structures. We give here a new presentation, based on  two types of deformations of Vaisman structures. 

In Section \ref{mainsect} we prove the main results.
Given an lcK structure $(\Omega, \theta)$ with holomorphic Lee  vector field $T$ on a compact complex manifold $(M,J)$ of Vaisman type, we show in a first step that there exists a Vaisman structure $(g_0,\Omega_0,\theta_0)$ on $(M,J)$ adapted to this lcK structure, in the sense that its Lee form $\theta_0$ is cohomologous to $\theta$ and is $JT$-invariant. This follows by an averaging process and using the aforementioned deformations of Vaisman structures, as well as the convexity of the space of Lee forms of Vaisman structures. In particular, the anti-Lee vector field $JT$ is holomorphic  and Killing with respect to the Vaisman metric $g_0$. The space of all holomorphic Killing vector fields on a compact Vaisman manifold will be described as a separate result (cf. Lemma \ref{holkill}), which allows us to deduce that $JT$ is completely determined by a function $a$, via the formula $JT=aJT_0-J\mathrm{grad}^{g_0}a$. Moreover, we prove that the function $a$ is necessarily positive (see Theorem  \ref{lckpothol}).

Conversely, we show that on a compact Vaisman manifold $(M,J,g_0,\Omega_0,\theta_0)$, given a holomorphic Killing vector field of the form $K=aJT_0-J\mathrm{grad}^{g_0}a$, where $a$ is a positive function, there exists an lcK structure with holomorphic Lee vector field equal to $-JK$. Moreover, we describe all these lcK structures by two parameters, namely a $K$-invariant function and a $K$-invariant twisted harmonic form of type $(0,1)$, which satisfy two certain positivity conditions, thus defining an open set in some infinite-dimensional vector space (see Theorem \ref{mainthm}).

As a last result, we show that on a compact lcK manifold, not necessarily of Vaisman type,
if we additionally assume that the lcK structure with holomorphic Lee vector field $T$ has a potential, then there exists a $JT$-invariant potential which is positive everywhere.

{\bf Acknowledgments.} This work was supported by the Procope Project No. 57445459 (Germany) /  42513ZJ (France).

\section{Preliminaries on lcK and Vaisman manifolds}

Let $(M,J)$ be a compact complex manifold. An lcK structure on $(M,J)$ is a Hermitian metric $g$ whose fundamental form $$\Omega:=g(J\cdot,\cdot)$$ satisfies the lcK condition $\di \Omega=\theta\wedge\Omega$, for some closed 1-form $\theta$ called the Lee form. The metric dual $T$ of $\theta$ will be called the Lee vector field. 

The lcK condition is conformally invariant: if $(\Omega,\theta)$ is an lcK structure on $(M,J)$, then $(e^f\Omega,\theta+\di f)$ is an lcK structure on $(M,J)$ for every smooth function $f$. An lcK structure is thus globally conformally Kähler if and only if its Lee form is exact. If this does not hold, the lcK structure is called {\em strict}.

On a manifold $M$ endowed with a closed 1-form $\theta$, one may introduce the twisted differential $\di_\theta:=\di-\theta\wedge\cdot$, which satisfies $\di_\theta^2=0$, and thus defines so-called Morse-Novikov (or twisted) cohomology groups $H^*_\theta(M)$.
If $(M,J)$ is a complex manifold, the twisted differential  can be decomposed as follows: 
$$\di_\theta=\partial_\theta+\bar\partial_\theta,$$ with
$$\partial_\theta:=\partial-\theta^{1,0}\wedge\cdot \ ,\qquad\bar\partial_\theta:=\bar\partial-\theta^{0,1}\wedge\cdot\ ,$$
where $\theta^{1,0}:=\frac{1}{2}(\theta+iJ\theta)$ and $\theta^{0,1}:=\frac{1}{2}(\theta-iJ\theta)$. One may also introduce the real operator $$\di^c_\theta:=i(\bar\partial_\theta-\partial_\theta)=[J,\di_\theta],$$ where $J$ acts on exterior forms as a derivation. Since $\di_\theta^2=0$, we immediately get
$$\partial_\theta^2=0,\qquad \bar\partial_\theta^2=0,\qquad \partial_\theta\bar\partial_\theta+\bar\partial_\theta\partial_\theta=0,\qquad \di_\theta\di^c_\theta=2i\partial_\theta\bar\partial_\theta.$$

Note that, by definition, a Hermitian form $\Omega$ on $(M,J)$ is lcK if and only if there exists a closed 1-form $\theta$ such that $\di_\theta\Omega=0$. The lcK structure $(\Omega,\theta)$ is called $\di_\theta$-exact if there exists a $1$-form $\beta$ such that $\Omega=\di_\theta\beta$.

If $\theta':=\theta+\di f$, then $\di_{\theta'}=e^f\di_\theta e^{-f}$. Similar conjugation relations hold for the other operators introduced above, so that their properties only depend on the cohomology class of~$\theta$. In particular, for second order operators the following relation holds:
	\begin{equation}\label{reltwisteddiff}
	\mathrm{d}_{\theta'}\mathrm{d}_{\theta'}^c(e^{f} h)=e^f\mathrm{d}_{\theta}\mathrm{d}_{\theta}^c h, \quad \text{ for any } h\in\mathcal{C}^\infty(M).
	\end{equation}

\begin{defn} 
An lcK metric $g$ on $(M,J)$ is called Vaisman if its Lee vector field $T$ is Killing with respect to~$g$ and non-zero. A complex manifold $(M,J)$ is called of Vaisman type if it admits a Vaisman structure.
\end{defn}

Let $\nabla$ denote the Levi-Civita connection of $g$. Since $\theta$ is closed, $\nabla \theta$ is a symmetric bilinear form. Moreover, the Lie derivative of $g$ with respect to $T$ is equal to the symmetric part of $\nabla\theta$. The above condition that $T$ is a Killing vector field is thus equivalent to the more familiar condition that $\nabla \theta=0$. 

Assume from now on that $(M,J,g_0, \Omega_0,\theta_0)$ is Vaisman. Being parallel, $\theta_0$ has constant norm with respect to $g_0$. Moreover, $T_0$ is  holomorphic with respect to $J$, i.e. $\cL_{T_0}J=0$ (cf. \cite{v} or \cite[Lemma 3]{mmo}). Hence, it also follows that $\cL_{T_0}\Omega_0=0$, so using the Cartan formula we get 
$$0=\cL_{T_0}\Omega_0=\di(T_0\lrcorner\Omega_0)+T_0\lrcorner\di\Omega_0=\di J\theta_0+T\lrcorner(\theta_0\wedge\Omega_0)=\di J\theta_0+|\theta_0|^2\Omega_0-\theta_0\wedge J\theta_0,$$
whence the well known formula
\begin{equation}\label{djt}
\di J\theta_0=\theta_0\wedge J\theta_0 -|\theta_0|^2\Omega_0,
\end{equation} 
which can also be written as 
\begin{equation}\label{pv}\Omega_0=d_{\theta_0}d^c_{\theta_0}\left(\frac1{|\theta_0|^2}\right).
\end{equation}  

\begin{defn}\label{potential} An lcK structure $(J,\Omega,\theta)$ for which there exists a function $h$ such that $\Omega=d_{\theta}d^c_{\theta}(h)$ is called an lcK structure with potential \cite{ov2010}. 
\end{defn}

The above formula \eqref{pv} shows that Vaisman metrics have potential, and by \eqref{reltwisteddiff}, if $(\Omega,\theta)$ is an lcK structure with potential $h$, then $(e^f\Omega,\theta+\di f)$ is an lcK structure with potential $e^fh$. 

\begin{remark}\label{strict}
Every lcK metric with potential, in particular every Vaisman metric, on a compact complex manifold is strict. Indeed, if $\theta=-\di f$ and $\Omega=\mathrm{d}_{\theta}\mathrm{d}_{\theta}^c h$, then \eqref{reltwisteddiff} shows that $e^f\Omega=\di\di^c(e^fh)$ is a Kähler metric with global potential, which is not possible on compact manifolds.
\end{remark}

Note that there exist lcK metrics without potential (cf. \cite{ov2010}), even on Vaisman-type manifolds (cf. \cite{goto}). 
However, N. Istrati has recently shown the following result which will be very useful for us in the sequel.

\begin{prop}\label{decomposition} (cf. \cite[Theorem 6.2]{ni})
	Let $(M,J,g_0,\Omega_0,\theta_0)$ be a compact Vaisman manifold and let $(\Omega, \theta_0)$ be an lcK structure on $(M,J)$ with the same Lee form. Then there exists a function $h\in \mathcal{C}^\infty(M, \mR)$ and a form $\alpha\in\Omega^{0,1}(M)$ in the (finite dimensional) kernel of the twisted Laplacian $\Delta_{\bar\partial_{\theta_0}}:=\bar\partial_{\theta_0}\bar\partial_{\theta_0}^*+\bar\partial_{\theta_0}^*\bar\partial_{\theta_0}$, such that:
	\begin{equation}\label{decomp}
	\Omega=d_{\theta_0}d^c_{\theta_0}h+\partial_{\theta_0}\alpha+\bar\partial_{\theta_0}\overline\alpha,
	\end{equation}
	where $\bar\partial_{\theta_0}^*$ denotes the formal adjoint of $\bar\partial_{\theta_0}$ with respect to $g_0$.
\end{prop}

The proof is based on standard Hodge theory, using the identification of the twisted Bott-Chern cohomology with some twisted Dolbeault cohomology, due to the vanishing of the Morse-Novikov cohomology on Vaisman manifolds (cf. Corollary \ref{vt} below).

\section{Deformations of Vaisman structures}\label{sectdeform}

In this section we consider two types of deformations of Vaisman structures, which have already been introduced in \cite{ts}. However, we present them from a different point of view that is useful for the purpose of this paper.

Recall first that the space of harmonic forms on the compact Vaisman manifold $(g_0, \Omega_0,\theta_0)$ decomposes as follows:
\begin{equation}\label{decv}\mathcal{H}^1(M,g_0)=\mathrm{span}\{\theta_0\}\oplus \mathcal{H}^1_0(M,g_0),
\end{equation} 
where $\mathcal{H}^1_0(M)$  is $J$-invariant and consists of harmonic 1-forms pointwise orthogonal to $\theta_0$ and $J\theta_0$.  In particular, the Cartan formula shows that $\cL_{T_0}\alpha=0$ and $\cL_{JT_0}\alpha=0$ for every $\alpha\in\mathcal{H}^1(M,g_0)$. For a proof, see for instance  \cite[Lemma 5.2]{mmp}.

If $(g_0,\Omega_0,\theta_0)$ is a Vaisman structure on $(M,J)$ with Lee vector field $T_0$, then for every positive real number $t>0$, the triple $(g'_0:=tg_0,\Omega'_0:=t\Omega_0,\theta_0)$ is a new Vaisman structure with the same Lee form, whose Lee vector field is $T'_0=\frac1t T_0$. The square norms of the Lee fields with respect to the respective metrics are related by 
\begin{equation}\label{tt}|T'_0|_{g'_0}^2=\theta_0(T'_0)=\frac1t\theta_0(T_0)=\frac1t|T_0|_{g_0}^2.\end{equation} 

\begin{defn} A Vaisman metric is called {\em normalized} if the Lee form (or, equivalently, the Lee vector field) has norm 1.
\end{defn}

From \eqref{tt} it follows that every Vaisman metric is proportional to a normalized Vaisman metric, which is unique in its homothety class. By \eqref{pv}, every normalized Vaisman metric has constant potential 1.

\begin{lemma}\label{lpc}
Let $(M,J, g_0, \Omega_0,\theta_0)$ be a compact Vaisman manifold and let $(g, \Omega,\theta)$ be any lcK structure on $(M,J)$. If $[\theta]=t[\theta_0]+[\alpha]$, with $t\in\mathbb{R}$ and $\alpha\in \mathcal{H}^1_0(M,g_0)$ denotes the decomposition of $[\theta]$ with respect to \eqref{decv}, then $t>0$.
\end{lemma}
\begin{proof}
In the conformal class of $g$, there exists an lcK metric with Lee form equal to $t\theta_0+\alpha$. So  without loss of generality, we assume that $\theta=t\theta_0+\alpha$. As $\di J\alpha=0$, we get from \eqref{djt}
$$\di J\theta=t\di J\theta_0+\di J\alpha=t(\theta_0\wedge J\theta_0-|\theta_0|_{g_0}^2\Omega_0),$$ 
Let $n$ denote the complex dimension of $M$. Taking the wedge product with $\Omega^{n-1}$ and integrating over $M$ yields
$$\int_M\di J\theta\wedge\Omega^{n-1}=t\int_M(\theta_0 \wedge J\theta_0-|\theta_0|_{g_0}^2\Omega_0)\wedge \Omega^{n-1}, $$
which has the opposite sign of $t$, since $\int_M(\theta_0 \wedge J\theta_0-|\theta_0|_{g_0}^2\Omega_0)\wedge \Omega^{n-1}\le0$.
On the other hand, by Stokes' theorem, we have
$$\int_M\di J\theta\wedge\Omega^{n-1}=\int_M J\theta\wedge\di (\Omega^{n-1})=
-(n-1)\int_M \theta\wedge J\theta\wedge \Omega^{n-1}<0,$$ 
since $\theta\neq 0$. We thus deduce that $t>0$.
\end{proof}

\begin{lemma}\label{d1}
Let $(M,J, g_0, \Omega_0,\theta_0)$ be a Vaisman manifold. 
For every positive real number $t>0$ and harmonic $1$-form $\alpha\in \mathcal{H}_0^1(M,g_0)$, the pair
$(\Omega'_0,\theta'_0)$ defined by $\theta'_0:=t\theta_0+\alpha$ and $\Omega'_0:=\theta'_0\wedge J\theta'_0-\di J\theta'_0$ is a normalized Vaisman structure on $(M,J)$.
\end{lemma}
  
\begin{proof} 
Since $\theta_0$, $\mathcal{H}_0^1(M,g_0)$, and the expression of $\Omega'_0$ do not change if $\Omega_0$ is rescaled by a constant, one can assume that $(M,J, g_0, \Omega_0,\theta_0)$ is normalized. Then \eqref{djt} shows that $\di J\theta_0=\Omega_0-\theta_0\wedge J\theta_0$. Moreover $\di\alpha=\di J\alpha=0$, whence
\begin{equation}\label{op}
\Omega'_0=(t\theta_0+\alpha)
\wedge J (t\theta_0+\alpha)+ t(\Omega_0-\theta_0\wedge J\theta_0).
\end{equation}
Denoting as before by $T_0:=\theta_0^\sharp$ the metric dual of $\theta_0$ with respect to $g_0$, we have that the kernel of $\Omega_0-\theta_0\wedge J\theta_0$ is the span of $T_0$ and $JT_0$. On the other hand, 
$$[(t\theta_0+\alpha)
\wedge J (t\theta_0+\alpha)](aT_0+bJT_0,J(aT_0+bJT_0))=t^2(a^2+b^2)$$ since $\alpha$ vanishes pointwise on $T_0$ and $JT_0$. This shows that the right-hand side of \eqref{op}
is positive definite since it is the sum of two positive semi-definite (1,1)-forms which have no common kernel. Clearly $\di\Omega'_0=\theta'_0\wedge\Omega'_0$, so 
$(\Omega'_0,\theta'_0)$ is lcK.

Moreover, $T_0\lrcorner \Omega'_0= t J\theta'_0$, so $T'_0:=\frac{1}{t}T_0$ is the Lee vector field of $(\Omega'_0,\theta'_0)$. In particular, $T'_0$ is holomorphic, and furthermore $\cL_{T'_0}\Omega'_0=\frac{1}{t}\cL_{T_0}\Omega'_0=0$ since $\cL_{T_0}\alpha=0$, $\cL_{T_0} J=0$, and $\cL_{T_0} \theta_0=0$. Hence $T'_0$ is Killing with respect to $g'_0$, and thus $(\Omega'_0,\theta'_0)$ is Vaisman. Finally, $\theta'_0(T'_0)=t\theta_0+\alpha(\frac1t T_0)=\theta_0(T_0)=1$, so $(\Omega'_0,\theta'_0)$ is normalized.
\end{proof}

On a complex manifold $(M,J)$, the Vaisman structure $(\Omega'_0,\theta'_0)$ defined as in Lemma~\ref{d1} will be called a {\sl deformation of type I} of $(\Omega_0,\theta_0)$. 

A direct consequence of Lemmas \ref{lpc} and \ref{d1} is the following result:

\begin{proposition}\label{deformI}
The Lee form of any lcK structure on a Vaisman manifold is cohomologous to the Lee form of a Vaisman structure which is a deformation of type I of the initial Vaisman structure.
\end{proposition}

Thus, if we denote by $H^1_{\mathrm{lcK}}(M)$ and by $H^1_{\mathrm{Vaisman}}(M)$   the set of all de Rham cohomology classes that are represented by the Lee form of an lcK structure, respectively of a Vaisman structure, then on a compact complex manifold of Vaisman type these sets are equal and are described as follows with respect to an arbitrarily fixed Vaisman structure $(g_0,\Omega_0, \theta_0)$:
\[ H^1_{\mathrm{lcK}}(M)=H^1_{\mathrm{Vaisman}}(M)=\{ t[\theta_0]+[\alpha]\in H^1_{\mathrm{dR}}(M, \mR)\, |\, t>0, \alpha \in\mathcal{H}_0^1(M,g_0)\}.\]

\begin{cor}\label{vt} If $\theta$ is the Lee form of some lcK structure on a Vaisman-type manifold $(M,J)$, the twisted cohomology $H^*_\theta(M)$ vanishes.
\end{cor}

\begin{proof}  By Proposition \ref{deformI}, $\theta$ is cohomologous to the Lee form $\theta_0$ of a Vaisman structure on $(M,J)$. Since the twisted cohomology only depends on the cohomology class of the $1$-form, we have $H^*_\theta(M)=H^*_{\theta_0}(M)$. On the other hand, it was shown in \cite{llmp} that the twisted cohomology $H^*_{\theta_0}(M)$ vanishes whenever $\theta_0$ is non-vanishing and parallel with respect to some Riemannian metric on $M$.
\end{proof}

% Deformations of type II of Vaisman structures
We now consider another type of deformation of Vaisman structures, which preserves the cohomology class of the Lee form. Let $(g_0, \Omega_0,\theta_0)$ be a Vaisman structure on $(M,J)$ with Lee vector field $T_0$. Let $f\in C^\infty(M)$, such that $T_0(f)=JT_0(f)=0$. We define the closed $1$-form $\theta'_0$ and the $(1,1)$-form $\Omega'_0$ by
\begin{eqnarray}
\theta'_0 &= & \theta_0+\di f,\\
\Omega'_0 &= & \di_{\theta'_0} (-J\theta'_0)=|\theta_0|_{g_0}^2\Omega_0+\theta_0\wedge J\di f+\di f\wedge J\theta_0+\di f\wedge J\di f -\di\di^c f.
\end{eqnarray}
When $\Omega'_0$ is a positive $(1,1)$-form (for instance if $f$ is close to $0$ in the $C^2$ sense), the pair $(\Omega'_0,\theta'_0)$ is an lcK structure on $(M,J)$.

Moreover, we have $T_0\lrcorner\Omega'_0=|\theta_0|_{g_0}^2J\theta_0+|\theta_0|_{g_0}^2J\di f=|\theta_0|_{g_0}^2J\theta'_0$, since $T_0\lrcorner \di f=T_0\lrcorner J\di f=0$ and $T_0\lrcorner\di\di^cf=\cL_{T_0}\di^c f=\di^c\cL_{T_0} f=0$. Therefore, the Lee vector field of $(\Omega'_0,\theta'_0)$ is $T'_0=\frac1{|\theta_0|_{g_0}^2}T_0$. Since $\cL_{T_0}J=0$ and $\cL_{T_0}\Omega'_0=0$, and $\theta'_0(T'_0)=1$, the structure $(\theta'_0,\Omega'_0)$ is normalized Vaisman and we will call it a {\sl deformation of type II} of $(\Omega_0,\theta_0)$.

The next result is well-known (cf. \cite{ts}); we provide here a short proof for convenience.

\begin{proposition}\label{proportionalLee}
Let $(g_0, \Omega_0,\theta_0)$ and $(g'_0,\Omega'_0,\theta'_0)$ be two Vaisman structures on a compact complex manifold $(M,J)$, with Lee vector fields $T_0$, respectively $T_0'$. Then there exists a positive constant $\lambda>0$, such that $T'_0=\lambda T_0$. 
\end{proposition}

\begin{proof}
Since a constant rescaling of the fundamental form induces a constant rescaling of the Lee vector field, one can assume that the two Vaisman structures are normalized.
The closed $1$-form $\theta'_0$ is cohomologous to a harmonic $1$-form with respect to the metric $g_0$. By  \eqref{decv}, there exist $t\in\mR$  and $\alpha\in \mathcal{H}^1_0(M,g)$ such that  $[\theta'_0]=[t\theta_0+\alpha]$. Lemma~\ref{lpc} then yields that the real number $t$ is positive, and by Lemma \ref{d1} it follows that $t\theta_0+\alpha$ is the Lee form of a normalized Vaisman structure, whose Lee vector field is $\frac{1}{t}T_0$. 

Thus, without loss of generality, we may assume that $[\theta'_0]=[\theta_0]$, so there exists $f\in C^\infty(M)$, such that $\theta'_0=\theta_0+\di f$.
The $(n,n)$-form $(-1)^n\di J\theta_0\wedge (\di J\theta'_0)^{n-1}$ is exact and semi-positive, hence
\begin{equation}\label{van}\di J\theta_0\wedge (\di J\theta'_0)^{n-1}=0
\end{equation}
by Stokes' formula. The interior product of $\di J\theta'_0$ with $T'_0$ and $JT'_0$ vanishes, and $(\di J\theta'_0)^{n-1}$ is nowhere vanishing. Thus, taking the interior product with $T'_0$ and $JT'_0$ in \eqref{van} yields
$$ 0=\di J\theta_0(T_0',JT_0'),$$
whence $T'_0=cT_0+c'JT_0$, for some functions $c$ and $c'$, which have to be constant, since $T_0$ and $T'_0$ are holomorphic and $M$ is compact. Moreover, 
\begin{gather*}
0=\cL_{T'_0}\theta'_0=c\cL_{T_0}\theta'_0+c'\cL_{JT_0}\theta'_0=c\di(T_0(f))+c'\di (JT_0(f)),\\
0=\cL_{JT'_0}\theta'_0=c\cL_{JT_0}\theta'_0-c'\cL_{T_0}\theta'_0=-c'\di(T_0(f))+c\di (JT_0(f)).
\end{gather*}
hence $T_0(f)=JT_0(f)=0$. Therefore $(\Omega'_0,\theta'_0)$ is obtained from $(\Omega_0,\theta_0)$ by a deformation of type II, which implies that $T_0=sT'_0$ for some positive real number $s$.
\end{proof}

Proposition~\ref{proportionalLee} allows to identify on a compact complex manifold of Vaisman type $(M, J)$ two naturally oriented 1-dimensional distributions, $\mathcal{T}$ spanned by $T_0$ and $J\mathcal{T}$ spanned by $JT_0$, where $T_0$ is the Lee vector field of some Vaisman structure on $(M,J)$, which also determines the orientation.  The $2$-dimensional distribution $\mathcal{T}\oplus J\mathcal{T}$ is called the \emph{canonical distribution}. We denote by $\mathcal{T}^+$ the subset of $\mathcal{T}$ consisting of the union of all positive half-lines in $\mathcal{T}$, namely $\mathcal{T}^+:=\underset{p\in M}{\cup} \mR^+T_0(p)$ and correspondingly we denote by $J\mathcal{T}^+:=\underset{p\in M}{\cup} \mR^+JT_0(p)$. Proposition~\ref{proportionalLee} ensures that $\mathcal{T}^+$ and $J\mathcal{T}^+$ are well-defined, independently of the choice of the Vaisman structure.

Moreover, the proof of Proposition~\ref{proportionalLee} actually shows the following:

\begin{proposition}
Let $(M,J, \Omega_0, \theta_0)$ be a compact Vaisman manifold. Then any normalized Vaisman structure $(\Omega'_0,\theta'_0)$ on $(M,J)$ is obtained by deformations of type I and II starting from the given Vaisman structure $(\Omega_0,\theta_0)$.
\end{proposition}

We now introduce the subspace $\cL\subset \Omega^1(M)$ of Lee forms of Vaisman structures, defined by
\[\cL:=\{\theta\in\Omega^1(M)\, |\, \di\theta=0\ \text{and}\ (\Omega:=\theta\wedge J\theta -  \di(J\theta),\theta) \text{ is a Vaisman structure}\},\]
which is tautologically in bijection with the set of normalized Vaisman structures on $(M,J)$.

\begin{lemma}\label{setL}
	Let $(M,J)$ be a Vaisman-type manifold. Then $\cL$
	is a convex cone inside $\Omega^1(M)$.	
\end{lemma}

\begin{proof}
	Let $\theta$ be the Lee form of a Vaisman structure $(\Omega,\theta)$ on $(M,J)$. For any $t>0$, the $1$-form $t\theta$ is the Lee form of the Vaisman structure $(t^2\theta\wedge J\theta-t\mathrm{d} J\theta,t\theta)$ obtained from $(\Omega,\theta)$ by a deformation of type I. Thus, $\cL$ is a cone inside $\Omega^1(M)$.
	
	We now show that $\cL$ is also a  convex set. Let $\theta_1, \theta_2\in \cL$. Thus $(\Omega_1:=\theta_1\wedge J\theta_1 -  \di(J\theta_1), \theta_1)$ and $(\Omega_2:=\theta_2\wedge J\theta_2 -  \di(J\theta_2), \theta_2)$ are two normalized Vaisman structures on $(M,J)$ with associated Riemannian metrics $g_1=\Omega_1(\cdot,J\cdot)$ and $g_2=\Omega_2(\cdot,J\cdot)$. We need to check that $\theta:=\theta_1+\theta_2$ also belongs to $\cL$. We define:
	\[ \Omega:=\theta\wedge J\theta-\mathrm{d} J\theta=(\theta_1+\theta_2)\wedge(J\theta_1+J\theta_2)- (\mathrm{d}J\theta_1+\mathrm{d}J\theta_2)=\Omega_1+\Omega_2+\theta_1\wedge J\theta_2+\theta_2\wedge J\theta_1\]
which is clearly a $(1,1)$-form. We claim that the symmetric tensor $g:=\Omega(\cdot,J\cdot)$ is positive definite.

	For any vector field $X$ on $M$ the following identity holds:
	\begin{equation}\label{eqX}
	\Omega(X,JX)=\Omega_1(X,JX)+\Omega_2(X,JX)+2\theta_1(X)\theta_2(X)+2J\theta_1(X)J\theta_2(X).
	\end{equation}
	
Using some local orthonormal bases $\{\theta_1, J\theta_1, \alpha_1, \dots, \alpha_{n-2}\}$ and $\{\theta_1, J\theta_1, \beta_1, \dots, \beta_{n-2}\}$ of $(\Lambda^1(M), g_1)$ and $(\Lambda^1(M), g_2)$ respectively, \eqref{eqX} becomes:
	\begin{equation}\label{eqX2}
	\Omega(X,JX)=(\theta_1(X)+\theta_2(X))^2+(J\theta_1(X)+J\theta_2(X))^2+\sum_{i=1}^{n-2}(\alpha_i(X)^2+\beta_i(X)^2).
	\end{equation}
	Thus, $\Omega$ is positive semi-definite. In order to show that $\Omega$ is actually positive definite, assume that $\Omega(X,JX)=0$ for some vector $X$. If we denote by $T_1$ and $T_2$ the Lee vector fields of the Vaisman structures $(\Omega_1, \theta_1)$, respectively $(\Omega_2, \theta_2)$, then the above equation yields that $X$ lies in the span of $\{T_1, JT_1\}$, respectively of $\{T_2, JT_2\}$. 
	
	By Proposition~\ref{proportionalLee}, there exists a positive constant $\lambda$, such that $T_1=\lambda T_2$. Thus, $X$  is of the following form: $X=aT_1+bJT_1=\lambda(aT_2+bJT_2)$. The assumption $\Omega(X,JX)=0$ together with \eqref{eqX2} also implies that $\theta_1(X)+\theta_2(X)=0$ and $J\theta_1(X)+J\theta_2(X)=0$. Hence, we get that $a(\lambda+1)=0$ and $b(\lambda+1)=0$. As $\lambda>0$, it follows that $a=b=0$ and thus $X=0$, showing that $\Omega$ is positive definite and $(\Omega, \theta)$ is a Vaisman structure on $(M,J)$. Let us remark that the Lee vector field $T$ of this Vaisman structure is $T=\frac{1}{1+\lambda} T_1=\frac{\lambda}{1+\lambda} T_2$.
%	
%	, as the following computation shows:
%	\[T_1\lrcorner\Omega=T_1\lrcorner((\theta_1+\theta_2)\wedge (J\theta_1+J\theta_2))=J\theta_1+\lambda J\theta_2+ J\theta_2+\lambda J\theta_1 =(1+\lambda)J(\theta_1+\theta_2)=(1+\lambda)J\theta.\]
\end{proof}

\section{LcK structures with holomorphic Lee vector field}\label{mainsect}

This section contains the main result of the paper, namely the description of all lcK structures with holomorphic Lee vector field on a compact Vaisman-type complex manifold. 

Let $(M,J)$ be  a compact complex manifold of Vaisman type, and let $\mathcal{T}$ denote as before the 1-dimensional distribution spanned by the Lee vector field of any Vaisman structure on $(M,J)$ (cf. Proposition \ref{proportionalLee}). We consider the space $\mathcal{HL}$ of holomorphic vector fields on $(M,J)$ of Lee type,  defined as follows:

\begin{defn}
The set $\mathcal{HL}(M,J)$ is the set of all holomorphic vector fields which can be obtained as the Lee vector field of some lcK structure on $(M,J)$.
\end{defn}

Let us introduce the following notion, which will be used in the sequel:
\begin{defn}
	A vector field $X$ on a manifold $M$ is called of Killing type if there exists a Riemannian metric $g$ on $M$, such that $X$ is a Killing vector field of $g$.
\end{defn}	

The first observation is that if $T\in\mathcal{HL}(M,J)$ then $JT$ is holomorphic and of Killing type. More precisely we have the following result:

\begin{lemma}\label{hk}
If $T$ is the Lee vector field of some lcK structure $(g,\Omega,\theta)$, then $\cL_{JT}\Omega=0$. If $T$ is holomorphic, then $JT$ is Killing for $g$.
\end{lemma}
\begin{proof}
Indeed, the Cartan formula shows that the anti-Lee vector field $JT$ of an lcK structure always preserves the fundamental 2-form:
$$\cL_{JT}\Omega=\di(JT\lrcorner\Omega+JT\lrcorner(\di\Omega)=\di\theta+JT\lrcorner(\theta\wedge\Omega)=0.$$

If, moreover, $T$ is holomorphic, as $J$ is integrable, $JT$ is also holomorphic, thus $\cL_{JT}g=\cL_{JT}(\Omega(\cdot,J\cdot))=0$.
\end{proof}

Let us now give the description of holomorphic Killing vector fields on Vaisman manifolds, which will be used in the sequel.

\begin{lemma}\label{holkill}
Let $(M,J,g_0,\Omega_0, \theta_0)$ be a normalized Vaisman manifold with Lee vector field $T_0$. Then any holomorphic Killing vector field $K$ on $(M,J,g_0)$ is of the following form:
\[K=cT_0+aJT_0+ K_0,\]
where $K_0\in\{T_0, JT_0\}^{\perp}$, $c$ is a constant, and the function $a\in\mathcal{C}^\infty(M)$, called the Hamiltonian of the holomorphic Killing vector field $K$, satisfies $K_0\lrcorner \Omega_0=\mathrm{d}a$ and $T_0(a)=JT_0(a)=0$.
	\end{lemma}

\begin{proof}
Let $K$ be a holomorphic Killing vector field on $(M,J,g_0)$. Then $K$ leaves also invariant the fundamental form, $\mathcal{L}_{K}\Omega_0=0$, and thus also the Lee form, $\mathcal{L}_{K}\theta_0=0$. From the Cartan formula and the closedness of the Lee form $\theta_0$, it follows that $c:=\theta_0(K)$ is constant. Hence, the vector field $K$ splits as:
\[K=cT_0+aJT_0+K_0,\]
where $K_0\in\{T_0, JT_0\}^{\perp}$ and $a:=J\theta_0(K)$. Furthermore, we compute again using Cartan's formula and \eqref{djt}:
\begin{equation*}
\begin{split}
0&=\mathcal{L}_{K} J\theta_0=K\lrcorner \mathrm{d}J\theta_0+\mathrm{d} (K\lrcorner J\theta_0)=K\lrcorner (\theta_0\wedge J\theta_0-\Omega_0)+\mathrm{d}a\\
&=cJ\theta_0-a\theta_0- K\lrcorner\Omega_0+\mathrm{d}a=cJ\theta_0-a\theta_0-cJ\theta_0+a\theta_0-K_0\lrcorner\Omega_0+\mathrm{d}a\\
&=\mathrm{d}a-K_0\lrcorner\Omega_0.
\end{split}
\end{equation*}
In particular, it follows that $T_0(a)=JT_0(a)=0$, since $\Omega_0(K_0, T_0)=\Omega_0(K_0, J T_0)=0$. 
	\end{proof}

\begin{lemma}\label{inv}
	Let $(M,J)$ be a compact manifold of Vaisman type and let $X$ be a holomorphic vector field of Killing type on $M$. Then for any class $[\theta]\in H^1_{\mathrm{lcK}}(M)$
	there exists a normalized Vaisman structure $(\Omega_0, \theta_0)$ on $(M,J)$ such that $[\theta_0]=[\theta]$ and $\mathcal{L}_X\theta_0=0$.
\end{lemma}

\begin{proof}
By Proposition~\ref{deformI}, there exists a Vaisman structure $(g_1, \Omega_1, \theta_1)$ on $(M,J)$, such that $[\theta_1]=[\theta]$. Let $g$ be a Riemannian metric on $M$ for which $X$ is a Killing vector field. Then the flow of $X$ consists of isometries $\varphi_t\in \mathrm{Iso}(M,g)$, so its closure $G:=\overline{\{\varphi_t\}}$ in $\mathrm{Iso}(M,g)$ is a compact torus with Haar measure $\mathrm{d}\mu$, normed such that $\int_G \mathrm{d}\mu=1$. We define 
\[\theta_0:=\int_{\gamma\in G} \gamma^*\theta_1\, \mathrm{d}\mu \in \Omega^1(M).\]
Then the following cohomology classes are equal: $[\theta_0]=[\theta_1]=[\theta]$, since the action of any connected Lie group is trivial in cohomology. Furthermore, the $1$-form $\theta_0$ is invariant under $X$, \emph{i.e.} $\mathcal{L}_{X}\theta_0=0$, since $\gamma^*\theta_0=\theta_0$, for any $\gamma\in G$. 
As $X$ is holomorphic, $\gamma^*\theta_1$ is the Lee form of the Vaisman structure $\gamma^*(\Omega_1)$ on $(M,J)$ for every $\gamma\in G$. By the convexity of the set of Lee forms of Vaisman structures (Lemma \ref{setL}), we obtain that $\theta_0$ is the Lee form of a normalized Vaisman structure whose fundamental $2$-form is $\Omega_0:=\theta_0\wedge J\theta_0-\mathrm{d}J\theta_0$. 	
\end{proof}		

\begin{corollary}\label{vaismaninv}
	Let $(M,J)$ be a compact manifold of Vaisman type and let $T\in\mathcal{HL}(M,J)$ be a holomorphic vector field of Lee type. Then for every lcK structure 
	$(g,\Omega,\theta)$ on $(M,J)$ whose Lee vector field is $T$, there exists a normalized Vaisman structure $(\Omega_0, \theta_0)$ on $(M,J)$ such that $[\theta_0]=[\theta]$ and $\theta_0(JT)=0$.
\end{corollary}

\begin{proof} 
By Lemma \ref{hk}, $JT$ is Killing for $g$. Then Lemma~\ref{inv} applied to $X:=JT$ and to the class $[\theta]$ ensures the existence of a normalized Vaisman structure $(\Omega_0, \theta_0)$ on $(M,J)$ such that $[\theta_0]=[\theta]$ and $\mathcal{L}_{JT}\theta_0=0$. We only need to check that any such Vaisman structure satisfies $\theta_0(JT)=0$. 
	
	Cartan's formula together with $\mathcal{L}_{JT}\theta_0=0$ already imply that $c:=\theta_0(JT)$ is constant. Moreover, since $\theta$ and $\theta_0$ are cohomologous, there exists $f\in\mathcal{C}^\infty(M, \mR)$ such that $\theta=\theta_0+df$ and thus we compute:
	\[0=\theta(JT)=(\theta_0+df)(JT)=c+JT(f).\]
	Considering this equality at a point of extremum for $f$ on the compact manifold $M$, it follows that the constant $c$ vanishes, hence $\theta_0(JT)=0$.
\end{proof}		

In the sequel, we will need the following criterion to decide when a vector field is the anti-Lee vector field of a $\di_\theta$-exact lcK structure:

\begin{lemma}\label{critantiLee}
	Let $(M,J, \Omega, \theta)$ be a strict lcK manifold, such that $\Omega=\di_{\theta}\beta$, for some $\beta\in\Omega^1(M)$. Let $K$ be a vector field on $M$ such that $\mathcal{L}_K\beta=0$ and $\theta(K)=0$. Then $K$ is the anti-Lee vector field of $(\Omega, \theta)$ if and only if $\beta(K)+1=0$. In particular, the Lee vector field of any lcK structure on such a manifold is not vanishing at any point.
\end{lemma}	
	
	\begin{proof}
		By definition, $K$ is the anti-Lee vector field of the lcK structure $ (\Omega, \theta)$ if and only if $K\lrcorner \Omega=-\theta$. On the other hand, we compute as follows, using Cartan's formula and the fact that $\mathcal{L}_K\beta=0$ and $\theta(K)=0$:
		\[K\lrcorner \Omega=K\lrcorner \di_{\theta}\beta=K\lrcorner (\di\beta-\theta\wedge \beta)=\mathcal{L}_K\beta-\di(\beta(K))-\theta(K)\beta+\beta(K)\theta=-\di_{\theta}(\beta(K)).\]
		Hence, $K\lrcorner \Omega=-\theta$ is equivalent to $\di_{\theta}(\beta(K)+1)=0$.
		
Since the lcK structure is not exact, $\di_{\theta}$ is injective on functions (cf. \cite[Lemma 2.1]{mmp}), showing that $K$ is the anti-Lee vector field of $(\Omega, \theta)$ if and only if $\beta(K)+1=0$. In particular, $K$, and thus also the Lee vector field $-JK$, is nowhere vanishing on $M$.
	\end{proof}	
	
	We will now gather some further information about holomorphic vector fields of Lee type on Vaisman-type manifolds.

\begin{theorem}\label{lckpothol}
Let $(M,J)$ be a  compact Vaisman-type manifold endowed with an lcK structure $(g,\Omega,\theta)$  whose Lee vector field $T$ is holomorphic. 

$(i)$ There exists a normalized Vaisman structure $(g_0,\Omega_0,\theta_0, T_0)$ on $(M,J)$ such that $[\theta]=[\theta_0]$ and $\theta_0(JT)=0$. 
	
$(ii)$ The anti-Lee vector field $JT$ is holomorphic and Killing with respect to the Vaisman metric $g_0$ and there exists a function $a\in\mathcal{C}^\infty(M)$, with $T_0(a)=JT_0(a)=0$, and such that $JT=aJT_0-J\mathrm{grad}^{g_0}a$.

$(iii)$ The function $a$ is everywhere positive on $M$.

Conversely, if $(g_0,\Omega_0,\theta_0, T_0)$ is a normalized Vaisman structure on $(M,J)$ and $K$ is a holomorphic Killing vector field on $(M,J,g_0)$ with positive Hamiltonian $a$ and vanishing $T_0$-component $c$, then $-JK$ is the Lee vector field of an lcK structure $(\Omega,\theta)$ on $(M,J)$ such that $[\theta]=[\theta_0]$ and $\theta_0(JT)=0$. 
\end{theorem}

\begin{proof} 
(i) Follows directly from Corollary~\ref{vaismaninv}.

(ii) Since $T$ is holomorphic, the same holds for $JT$, so $\mathcal{L}_{JT}J=0$. Moreover $\mathcal{L}_{JT}\theta_0=\di(\theta_0(JT))=0$, so $\mathcal{L}_{JT}\Omega_0=\mathcal{L}_{JT}(\theta_0\wedge J\theta_0-\di J\theta_0)=0$, whence $\mathcal{L}_{JT} g_0=0$.

By Lemma~\ref{holkill}, there exists a real constant $c$ and a function $a\in\mathcal{C}^\infty(M)$ (the Hamiltonian of $JT)$, with $T_0(a)=JT_0(a)=0$, and such that $JT=cT_0+aJT_0-J\mathrm{grad}^{g_0}a$. Moreover, $c$ vanishes, since
$$0=\theta_0(JT)=c-\theta_0(J\mathrm{grad}^{g_0}a)=c+JT_0(a)=c.$$

(iii) Corollary \ref{vt} ensures that every lcK structure $(\Omega,\theta)$ on a Vaisman-type manifold is $\di_\theta$-exact, i.e. $\Omega=\di_\theta\beta$ for some $1$-form $\beta$. Using an averaging argument as before, one can assume that $\mathcal{L}_{JT}\beta=0$. By the converse of Lemma \ref{critantiLee} applied to $K=JT$ we thus obtain that $T$ is never vanishing.

	 Since $[\theta]=[\theta_0]$, there exists $f$ such that $\theta=\theta_0+\di f$. We now compute:
	\begin{eqnarray*}
	|T|_g^2&=&\theta(T)=\theta_0(T)+T(f)=\theta_0(aT_0-\mathrm{grad}^{g_0}a)+aT_0(f)-\di f( \mathrm{grad}^{g_0}a)\\
	&=&a+aT_0(f)-\langle \di f, \di a \rangle_{g_0},
	\end{eqnarray*}
	which together with the fact that $T$ vanishes nowhere on $M$, yields the following inequality:
	\begin{equation}\label{ineqthetaT}
	a+aT_0(f)-\langle \di f, \di a \rangle_{g_0}>0
	\end{equation}
	
	Let us denote by $m:=\displaystyle \min_M a$ the minimum of the function $a$ on the compact manifold $M$. Applied at any $p\in a^{-1}(m)$, the above inequality  yields $a(p)(1+T_0(f)(p))>0$. Hence, in particular, we have $m=a(p)\neq 0$. 
	
	We assume, by contradiction, that $m < 0$. The above inequality yields $1+T_0(f)(p)<0$ for all $p\in a^{-1}(m)$. As $T_0$ is a parallel vector field of constant length $1$ with respect to the Vaisman metric $g_0$, each integral curve of $T_0$ is a complete geodesic of $(M,g_0)$. Moreover, the restriction of the function $a$ along any integral curve of $T_0$ is constant, since  $T_0(a)=0$. Hence, along a complete geodesic $\gamma\colon \mR\to M$ starting at a point $p\in a^{-1}(m)$, the  following inequality holds: $T_0(f)(\gamma(t))<-1$, for all $t\in\mR$, which yields a contradiction, since the function $f$ is bounded. This proves that $m=\displaystyle \min_M a>0$, hence the function $a$ is positive on $M$.
	
	Conversely, let $K=aJT_0-J\mathrm{grad}^{g_0}a$ be a holomorphic Killing vector field on $(M,J,g_0)$, with  $T_0(a)=JT_0(a)=0$. We claim that the Lee vector field of the lcK structure $$(\Omega:=\frac{1}{a}\Omega_0,  \theta:=\theta_0-\frac{\di a}{a})$$ is $-JK$. Indeed, 
	$$(-JK)\lrcorner \Omega=(aT_0-\mathrm{grad}^{g_0}a)\lrcorner (\frac1a \Omega_0)=J\theta_0-\frac1aJ\di a=J\theta.$$
\end{proof}	

\begin{remark}\label{remexpl}
Theorem~\ref{lckpothol} gives rise to examples of lcK structures on compact Vaisman-type manifolds whose Lee vector field is not homothetic to the common Lee vector field of the Vaisman structures. More precisely, if there exists on a compact Vaisman  manifold $(M,J, \Omega_0, \theta_0, T_0)$ a Killing vector field $K$ which is not a linear combination of $T_0$ and $JT_0$, then, according to Lemma~\ref{holkill}, it decomposes as $K=cT_0+aJT_0-J\mathrm{grad}^{g_0}a$, where $c$ is a constant and the Hamiltonian function $a$ is not constant. As $T_0$ is also a Killing vector field for $g_0$, by subtracting $cT_0$ and adding $kJT_0$ to $K$, for any constant $\displaystyle k>|\min_M a|$, we obtain a Killing vector field whose Hamiltonian function $a+k$ is positive everywhere and has no component along $T_0$, namely $\widetilde K=(a+k)JT_0-J\mathrm{grad}^{g_0}(a+k)$. According to the converse part of Theorem~\ref{lckpothol}, there exists an lcK structure on $(M,J)$ whose Lee vector field equals $-J\widetilde K=(a+k)T_0-\mathrm{grad}^{g_0}(a+k)$, which is clearly not homothetic to $T_0$, because $a$ is not constant and $\mathrm{grad}^{g_0}(a+k)$ is orthogonal to $T_0$. Examples of compact Vaisman  manifolds admitting Killing vector fields  which are not everywhere tangent to the canonical distribution $\mathcal{T}\oplus J\mathcal{T}$ are provided for instance by compact toric Vaisman manifolds (see \cite{n}, \cite{mmp}, \cite{p}) or compact homogeneous Vaisman manifolds (see \cite{gmo}).
\end{remark}	

The above result is not completely satisfactory, as the description of the space $\mathcal{HL}(M,J)$ of holomorphic vector fields of Lee type on $(M,J)$ obtained in Theorem \ref{lckpothol} depends on the choice of some background Vaisman structure. However, it is possible to give a completely intrinsic description of the space $\mathcal{HL}(M,J)$ on compact Vaisman-type manifolds:

\begin{thm} Let $(M,J)$ be  a compact complex manifold of Vaisman type, and let $\mathcal{T}$ denote as before the 1-dimensional distribution spanned by the Lee vector field of any Vaisman structure on $(M,J)$. Then a holomorphic vector field $T$ belongs to $\mathcal{HL}(M,J)$ if and only if the following conditions are satisfied:

(i) $JT$ is of Killing type.
	
(ii) For any point $p\in M$, if $T_p\in \mathcal{T}_p\oplus J\mathcal{T}_p$, then $T_p\in \mathcal{T}^+_p$.

Moreover, if $T\in\mathcal{HL}(M,J)$, then for any cohomology class $\mu\in H^1_{\mathrm{lcK}}(M)$ there exists an lcK structure whose Lee form represents this class and whose Lee vector field is $T$.
\end{thm}

\begin{proof}
	If $T\in\mathcal{HL}(M,J)$, there exists an lcK structure $(g, \Omega,\theta)$ on $(M,J)$ whose Lee vector field is $T$. By Lemma \ref{hk}, $JT$ is Killing for $g$, so (i) is satisfied.
	
	By Theorem~\ref{lckpothol}, we may choose  a Vaisman structure $(g_0,\Omega_0,\theta_0, T_0)$ on $(M,J)$ such that $[\theta]=[\theta_0]$, $\theta_0(JT)=0$, and $JT$ is a Killing vector field with respect to the metric $g_0$ satisfying
	\[T=aT_0-\mathrm{grad}^{g_0}a,\]
	where the Hamiltonian function $a\in\mathcal{C}^\infty(M)$ is everywhere positive.
	
	If $p$ is any point in $M$ such that $T_p\in \mathcal{T}_p\oplus J\mathcal{T}_p$, then $T_p=a(p)(T_0)_p$, because $\mathrm{grad}^{g_0}a\in\{T_0, JT_0\}^{\perp}$. In particular, as $a>0$, it follows that $T_p\in \mathcal{T}^+_p$. Thus (ii) is satisfied too.
	
	Conversely, let $T$ be a holomorphic vector field satisfying (i) and (ii). By Lemma~\ref{inv}, for any cohomology class in $\mu\in H^1_{\mathrm{lcK}}(M)$, there exists a normalized Vaisman structure $(\Omega_0, \theta_0)$ on $(M,J)$ such that $[\theta_0]=\alpha$, and $\mathcal{L}_{JT}\theta_0=0$. Hence, $JT$ is a Killing vector field of the Vaisman metric $g_0$ and according to Lemma~\ref{holkill}, we can write
	\[ JT=cT_0+aJT_0-J\mathrm{grad}^{g_0}a,\]
	where $c$ is a constant and $a\in\mathcal{C}^{\infty}(M)$ with $T_0(a)=JT_0(a)=0$. Applying $J$ to this equality yields $T=-cJT_0+aT_0-\mathrm{grad}^{g_0}a$. 
	
	Let $p$ be a point of minimum of the function $a$. Then $T_p=-c(JT_0)_p+a(p)(T_0)_p\in \mathcal{T}\oplus J\mathcal{T}$. The condition $(ii)$ from the definition of $\mathcal{HL}$ implies that $T_p\in\mathcal{T}^+_p$, whence $c=0$ and $a(p)>0$. This shows that $T=aT_0-\mathrm{grad}^{g_0}a$ and the function $a$ is positive. The converse part of Theorem~\ref{lckpothol} shows that $T$ is the Lee vector field of the lcK structure $(\Omega:=\frac{1}{a}\Omega_0,  \theta:=\theta_0-\frac{\di a}{a})$, so $T\in\mathcal{HL}(M,J)$. Moreover, the Lee form of this lcK structure satisfies $[\theta]=[\theta_0]=\mu$.
\end{proof}	

Now, that we have intrinsically characterized the set $\mathcal{HL}(M,J)$ of all holomorphic vector fields which can occur as Lee vector fields of lcK structures on Vaisman-type manifolds, we would like to describe, for each fixed $T\in\mathcal{HL}(M,J)$, the set of all lcK structures admitting $T$ as Lee vector field. 
In order to do so, we will fix some cohomology class $\mu\in H^1_{\mathrm{lck}}(M,J)$ and apply Corollary \ref{vaismaninv} in order to choose a normalized Vaisman structure $(\Omega_0,\theta_0)$ whose Lee form satisfies $[\theta_0]=\mu$ and $\theta_0(JT)=0$.

Our aim is to describe all lcK structures with Lee vector field $T$, and with Lee form in the cohomology class $\mu$. If $(\Omega,\theta)$ is such an lcK structure, then $\theta=\theta_0+\di f$, for some function $f\in\mathcal{C}^{\infty}(M, \mR)$, so $(e^{-f}\Omega,\theta_0)$ is an lcK structure on $(M,J)$ with Lee form $\theta_0$. By Proposition \ref{decomposition}, there exists a function $h\in \mathcal{C}^\infty(M, \mR)$ and a $(0,1)$-form $\alpha\in\Omega^{0,1}(M)$ in the kernel of the twisted Laplacian $\Delta_{\bar\partial_{\theta_0}}$, such that
	\begin{equation}\label{fo}
	e^{-f}\Omega=d_{\theta_0}d^c_{\theta_0}h+\partial_{\theta_0}\alpha+\bar\partial_{\theta_0}\overline\alpha.
	\end{equation}
Since $\mathcal{L}_{JT}\theta=0$ and  $\mathcal{L}_{JT}\theta_0=0$, it follows that $\mathcal{L}_{JT}\di f=0$, hence $JT(f)$ is constant. As $f$ has critical points, this constant must be zero, so $JT(f)=0$.
	
	Let $\varphi_t$ denote the flow of the holomorphic Killing vector field $JT$. Since $JT$ preserves $J$, $\Omega$, $\theta$, $\theta_0$ and $f$, we obtain:
	\begin{equation*} e^{-f}\Omega=\varphi_t^*(e^{-f}\Omega)=d_{\theta_0}d^c_{\theta_0}(\varphi_t^*h)+\partial_{\theta_0}(\varphi_t^*\alpha)+\overline{\partial_{\theta_0}(\varphi_t^*\alpha)}.\end{equation*}
	Thus, after averaging over the compact torus $G:=\overline{\{\varphi_t\}}$, we may assume that both the function $h$ and the $1$-form $\alpha$ are  invariant under $JT$. 
	
	Using the fact that $\bar \partial_{\theta_0}\alpha=0=\partial_{\theta_0}\bar\alpha$ and the relation \eqref{reltwisteddiff} between $\di_{\theta}$ and $\di_{\theta_0}$, the fundamental $2$-form $\Omega$ can be expressed as follows:
\[ \Omega=e^f(\di_{\theta_0}\di^c_{\theta_0} h+\partial_{\theta_0}\alpha+\bar\partial_{\theta_0}\overline\alpha)=e^f\di_{\theta_0}[\di^c_{\theta_0}h+2\mathrm{Re}\alpha]=\di_{\theta}[e^f (\di^c_{\theta_0} h+2\mathrm{Re}\, \alpha)]=\di_{\theta} \beta,\]
where 
\begin{equation} \label{beta}\beta:=e^f (d^c_{\theta_0} h+2\mathrm{Re}\, \alpha)
\end{equation}
 is also $JT$-invariant.

Lemma \ref{critantiLee} applied to $K:=JT$ shows that $\beta(JT)+1=0$, which by \eqref{beta} is equivalent to 
\begin{equation} \label{eqantiLee}
	e^{-f}=-T(h)+h\theta_0(T)-2\mathrm{Re}(\alpha(JT)).
	\end{equation}

Conversely, suppose that $h\in \mathcal{C}^{\infty}(M, \mR)$ is a smooth function and $\alpha\in\Omega^{0,1}(M)$ is a $(0,1)$-form in the kernel of the twisted Laplacian $\Delta_{\bar\partial_{\theta_0}}$ such that:
\begin{enumerate}
\item $\mathcal{L}_{JT}h=0$ and $\mathcal{L}_{JT}\alpha=0$;
\item the right hand side of \eqref{eqantiLee} is positive on $M$;
\item the $(1,1)$-form $\di_{\theta_0}\di^c_{\theta_0} h+\partial_{\theta_0}\alpha+\bar\partial_{\theta_0}\overline\alpha$ is positive definite.
\end{enumerate}
Then if $f$ is defined by \eqref{eqantiLee}, the pair $(\Omega:=e^f(\di_{\theta_0}\di^c_{\theta_0} h+\partial_{\theta_0}\alpha+\bar\partial_{\theta_0}\overline\alpha), \theta:=\theta_0+\mathrm{d}f)$ is an lcK structure on $(M,J)$ whose Lee vector field is $T$ by Lemma \ref{critantiLee}.

Motivated by the above considerations, we now introduce for any holomorphic vector field of Lee type $T\in \mathcal{HL}(M,J)$ and for any normalized Vaisman structure $(g_0, \Omega_0, \theta_0,T_0)$,
%  with Lee vector field $T_0$ and let \[K:=aJT_0-J\mathrm{grad}^{g_0}a\]
% be a Killing holomorphic vector field (\emph{cf.} Lemma~\ref{holkill}) whose Hamiltonian function $a$ is positive on $M$. Then, since the Hamiltonian $a$ is invariant by $JT_0$ according to Lemma~\ref{holkill}, the following holds: 
%\[K(a)=aJT_0(a)=0.\]
the sets of functions and of twisted harmonic forms that are invariant by $JT$:
\begin{gather*}
 \mathcal{C}^\infty_{JT}(M):=\{h\in\mathcal{C}^\infty(M)\, | \, JT(h)=0\}, \\
 \mathcal{H}^{*,*}_{\bar\partial_{\theta_0}, JT}(M):=\{\alpha\in\Omega^{*,*}(M)\, | \, \Delta_{\bar\partial_{\theta_0}}\alpha=0, \mathcal{L}_{JT}\alpha=0\},
\end{gather*}
and we define the open subset $\mathcal{F}_{JT}$ of $ \mathcal{C}^\infty_{JT}(M)\times \mathcal{H}^{0,1}_{\bar\partial_{\theta_0}, JT}(M)$ consisting of all pairs $(h, \alpha)$ that satisfy the following two conditions:
\begin{equation}
\begin{cases}
\di_{\theta_0}\di^c_{\theta_0} h+\partial_{\theta_0}\alpha+\bar\partial_{\theta_0}\overline\alpha \text{ is a positive definite } \text{$(1,1)$-form}\\
-T(h)+h\theta_0(T)-2\mathrm{Re}(\alpha(JT))>0.
\end{cases}
\end{equation}

Note that the set $\mathcal{F}_{JT}$ is non-empty. Indeed, the pair $(1,0)$ belongs to $ \mathcal{F}_{JT}$, because the $(1,1)$-form $\di_{\theta_0}\di^c_{\theta_0} 1=\Omega_0$ is positive definite and the second inequality is fulfilled since $\theta_0(T)$ is the Hamiltonian of the Killing vector field $JT$, which is positive by Theorem \ref{lckpothol}. We can now state our main result:

\begin{theorem}\label{mainthm}
Let $T\in\mathcal{HL}(M,J)$ be a holomorphic vector field of Lee type on a Vaisman-type manifold, and $\mu\in H^1_{\mathrm{lck}}(M,J)$ be a fixed cohomology class of lcK type. Fix any normalized Vaisman structure $(\Omega_0,\theta_0)$ whose Lee form satisfies $[\theta_0]=\mu$ and $\theta_0(JT)=0$ (the existence of such a Vaisman structure is granted by Corollary \ref{vaismaninv}).
	There is a surjective map from the set $\mathcal{F}_{JT}$ to the set of all lcK structures $(\Omega, \theta)$ on $(M,J)$ having $[\theta]=\mu$ and the Lee vector field equal to $T$, given by:
	\begin{equation}\label{defmap}
	(h, \alpha) \mapsto (\Omega:=e^f(\di_{\theta_0}\di^c_{\theta_0} h+\partial_{\theta_0}\alpha+\bar\partial_{\theta_0}\overline\alpha), \theta:=\theta_0+\mathrm{d}f), 
	\end{equation}
	with $e^{-f}:=-T(h)+h\theta_0(T)-2\mathrm{Re}(\alpha(JT))$. 
	
	Moreover, two pairs $(h, \alpha), (\widetilde h, \widetilde\alpha)\in \mathcal{F}_{JT}$ are mapped to the same lcK structure if and only if there is a positive constant $c$, such that $\alpha=c\widetilde \alpha$ and $h-c\widetilde {h}_0$ is equal to the imaginary part of a function in  $\mathcal{H}^{0,0}_{\bar\partial_{\theta_0}, JT}(M)$.
\end{theorem}	

\begin{proof}
	The first part of the theorem was already proved above. It remains to determine under which circumstances two pairs $(h, \alpha), (\widetilde h, \widetilde\alpha)$ define the same lcK structure on $(M,J)$.
	
Assume that $(h, \alpha), (\widetilde h, \widetilde\alpha)$ are two pairs in $\mathcal{F}_{JT}$ that define the same lcK structure, \emph{i.e.} $\widetilde \Omega=\Omega$ and $\widetilde \theta=\theta$. The last equality is equivalent to $\di f=\di \widetilde f$, showing that there exists a positive constant $c$, such that $e^{-\widetilde f}=ce^{-f}$.Thus, $\widetilde \Omega=\Omega$ reads:
	\begin{equation*}
	i\partial_{\theta_0}\bar\partial_{\theta_0} (h-c\widetilde h)+\partial_{\theta_0}(\alpha-c\widetilde\alpha)+\bar\partial_{\theta_0}\overline{(\alpha-c\widetilde\alpha)}=0,
	\end{equation*}
	or equivalently, using the fact that $\alpha$ and $\widetilde\alpha$ are in the kernel of the twisted Laplacian $\Delta_{\theta_0}$, so in particular $\bar\partial_{\theta_0}\alpha=\bar\partial_{\theta_0}\widetilde\alpha=0$ and $\bar\partial^*_{\theta_0}\alpha=\bar\partial^*_{\theta_0}\widetilde\alpha=0$, as:
	\begin{equation}\label{eqomega1}
	\di_{\theta_0}[i\bar\partial_{\theta_0} (h-c\widetilde h)+\alpha-c\widetilde\alpha+\overline{\alpha}-c\overline{\widetilde\alpha}]=0,
	\end{equation}
	By Corollary \ref{vt}, the homology of $\di_{\theta_0}$ vanishes, so \eqref{eqomega1} is fulfilled if and only if there exists a function $b\in\mathcal{C}^{\infty}(M, \mC)$, such that $i\bar\partial_{\theta_0} (h-c\widetilde h)+\alpha-c\widetilde\alpha+\overline{\alpha}-c\overline{\widetilde\alpha}=\di_{\theta_0}b$, or equivalently, when considering on both sides the forms of type $(1,0)$, respectively $(0,1)$:
	\begin{equation}\label{eqomega2}
	\begin{cases}
	\overline{\alpha}-c\overline{\widetilde\alpha}=\partial_{\theta_0} b,\\
	\alpha-c\widetilde\alpha +i\bar\partial_{\theta_0} (h-c\widetilde h)=\bar\partial_{\theta_0} b.
	\end{cases}
	\Longleftrightarrow 
	\begin{cases}
	\alpha-c\widetilde\alpha=\bar\partial_{\theta_0} \bar b,\\
	\bar\partial_{\theta_0} \bar b +i\bar\partial_{\theta_0} (h-c\widetilde h)=\bar\partial_{\theta_0} b.
	\end{cases}
	\end{equation}
The last equation yields $\bar\partial_{\theta_0}(2\mathrm{Im}\, b-h+c\widetilde h)=0$. On the universal cover $\widetilde M$ of $M$, this equation translates into the condition that the real-valued function $(2\mathrm{Im}\, b-h+c\widetilde h)e^{-\varphi}$, where $\varphi$ is a primitive of $\theta_0$ on $\widetilde M$, is holomorphic. This is only possible if the function is constant, and because of the equivariance condition, this constant must actually vanish. Hence, we obtain that $h-c\widetilde h=2\mathrm{Im}\, b$.\\
On the other hand, since $\alpha$ and $\widetilde\alpha$ are in the kernel of $\Delta_{\theta_0}$, the first equation in \eqref{eqomega2} implies
\begin{equation}\label{eqb}
 \bar\partial_{\theta_0}^*\bar\partial_{\theta_0} \bar b=\bar\partial_{\theta_0}^*(\alpha-c\widetilde\alpha)=0.
 \end{equation}
Taking the scalar product with $\bar b$ and integrating over the compact manifold $M$ yields $\bar\partial_{\theta_0} \bar b=0$, which together with the first equation in \eqref{eqomega2} shows that $\alpha-c\widetilde\alpha=0$. We thus also obtain that $h-c\widetilde h=\mathrm{Im}\, (-2\bar b)$, with $-2\bar b\in \mathcal{H}^{0,0}_{\bar\partial_{\theta_0}, JT}(M)$. 
\end{proof}

Our last result concerns lcK structures with holomorphic Lee vector field on complex manifolds which are not necessarily of Vaisman type, but assuming instead that the lcK structure has a potential (cf. Definition \ref{potential}).

\begin{proposition}
	Let $(M,J, g,\Omega, \theta)$ be a compact lcK manifold with potential and with holomorphic Lee vector field $T$. Then $T$ is not vanishing at any point of $M$ and the following assertions hold:
	
		$(i)$ There exists a $JT$-invariant potential $h$.
		
		$(ii)$  Any $JT$-invariant potential $h$ satisfies the equation $T(h)-h\theta(T)+1=0$.
	
		$(iii)$  Any $JT$-invariant potential $h$ is positive.

\end{proposition}

\begin{proof}
	
 The fact that $T$ is not vanishing at any point of $M$ is a direct consequence of 
Lemma~\ref{critantiLee}, as any lcK structure with potential on a compact manifold is strict (\emph{cf.} Remark~\ref{strict}).

	$(i)$  Let $h_0$ be a potential of the lcK structure, \emph{i.e.} $\Omega=\mathrm{d}_{\theta}\di_{\theta}^c h_0$. Since $\mathcal{L}_{JT}\theta=0$ and $\mathcal{L}_{JT}\Omega=0$, the flow $\varphi_t$ of $JT$ preserves $J$, $\theta$ and $\Omega$ and hence $\varphi_t^* h_0$ is also a potential of $(\Omega,\theta)$, for all $t$. Therefore, denoting as before by $G:=\overline{\{\varphi_t\}}$ the closure of the flow of $JT$ in $\mathrm{Iso}(M,g)$, the function $h:=\int_G \varphi_t^*h_0\, \mathrm{d}\mu$ is a $JT$-invariant potential.
	
$(ii)$ If the potential $h$ is $JT$-invariant, then we may apply  Lemma~\ref{critantiLee} to $\beta:=\di_{\theta}^c h$
 and obtain that $JT$ satisfies $\beta(JT)+1=0$, which is equivalent to $T(h)-h\theta(T)+1=0$.

$(iii)$ If $h$ be a $JT$-invariant potential, then by (ii), $h$ satisfies $T(h)-h\theta(T)+1=0$. If $p$ is a minimum of $h$ on $M$, then this equation implies that $h(p)=\frac{1}{\theta_p(T(p))}=\frac{1}{|T(p)|_g^2}>0$, hence $h>0$ on $M$.
\end{proof}

\end{document}